\documentclass[11pt]{article}
\usepackage{amsmath}
\usepackage{amsthm}
\usepackage{amssymb}
\newcommand{\dbar}{\bar{\partial}}
\newcommand{\C}{\mathbb{C}}
\newcommand{\R}{\mathbb{R}}
\newcommand{\Z}{\mathbb{Z}}

\theoremstyle{plain}
\newtheorem{theorem}{Theorem}[section]
\newtheorem{proposition}[theorem]{Proposition}
\newtheorem{lemma}[theorem]{Lemma}
\newtheorem{Corollary}[theorem]{Corollary}

\theoremstyle{definition}
\newtheorem{definition}[theorem]{Definition}
\newtheorem{remark}[theorem]{Remark}

\makeatletter
\@addtoreset{equation}{section}
\makeatother
\newcounter{constantLABEL}
\newcommand{\cref}[1]{C_{\ref{#1}}}

\newcounter{constantslabel}
\begin{document}


\title{
Seiberg--Witten type equations on compact symplectic 6-manifolds 
} 
\author{  
Yuuji Tanaka }
\date{}


\maketitle


\begin{abstract}
In this article, we consider a gauge-theoretic equation on compact
symplectic 6-manifolds, which 
 forms an elliptic system after gauge fixing. 
This can be thought of as a higher-dimensional
analogue of the Seiberg--Witten equation.   
By using the virtual neighbourhood method by Ruan \cite{R}, 
we define an integer-valued invariant, 
a 6-dimensional Seiberg--Witten invariant, from the moduli space of solutions to the
equations, 
assuming that the moduli space is compact; 
and it has no reducible solutions. 
We prove that the moduli spaces are compact if 
the underlying manifold is a compact K\"{a}hler
threefold. We then compute the integers in some cases. 
\end{abstract}



\footnote[0]{\textit{AMS 2010 Mathematics Subjet Classification}: 53C07. 
}

\section{Introduction}

Let $X$ be a compact symplectic 6-manifold with symplectic form
$\omega$. 
We take an almost complex structure $J$ compatible with 
the symplectic form $\omega$. 
We fix a $Spin^c$-structure $s$ on $X$. 
We denote the characteristic line bundle for $s$ by $\xi$. 
Then there exists a line bundle $L$ such that $\xi = L^2 \otimes
K_{X}^{-1}$, 
where $K_{X}^{-1}$ is the anti-canonical bundle of $X$. 
Let $A'$ be a connection on $\xi =L^{2} \otimes  K_{X}^{-1}$. 
We write $A' = A_{c} + 2A$, where $A_c$ is the canonical connection on
$K_{X}^{-1}$, which is fixed, and $A$ is a connection of a line bundle $L$. 
We then consider the following equations on compact symplectic 6-manifolds  
(see Section \ref{sec:mi} for more detail), 
seeking for  
a connection $A$ of $L$, $u \in \Omega^{0,3} (X)$, 
$\alpha \in C^{\infty}(L)$ 
and $\beta \in \Omega^{0,2} ( L)$.  
\begin{gather*}
  \dbar_{A} \alpha + \dbar_{A}^{*} \beta = 0 , \quad 
 \dbar_{A} \beta =  - \frac{1}{2} \alpha u ,  \\ 
 F_{A'}^{0,2} + \dbar^{*} u = \frac{1}{4} \bar{\alpha} \beta , \quad 
 \Lambda F_{A'}^{1,1} = - \frac{i}{8} \left( | u |^2 + |\beta|^2 -
 |\alpha|^2 \right),   
\end{gather*}
where $\Lambda = (\omega \wedge )^{*}$.

\begin{remark}
Richard P. W. Thomas once considered similar equations in \cite{T}. 
Our equations partially emerged out of discussion with Dominic Joyce around the
 end of 2010 together with the computation in the proof of Proposition \ref{eqk}. 
\end{remark}

These equations from an elliptic system with gauge fixing condition. 
We expect they enjoy nice properties similar to the
original Seiberg--Witten equations such as the compactness of the moduli
space.

In this article, we prove that the moduli spaces are compact if the
underlying manifold is a compact K\"{a}hler threefold. 
Hence, in this case, one can define an integer $n_{X} (s)$ for a
$Spin^c$-structure $s$, a 6-dimensional Seiberg--Witten invariant, 
by using Ruan's virtual neighbourhood method \cite{R}, if there are
no reducible solutions. 
We compute the numbers in some cases as follows.  
These are analogues of those for the Seiberg--Witten invariants in 4 dimensions. 
Firstly, we have the following.

\begin{theorem}[Corollary \ref{cor:kxn}]
Let $X$ be a compact K\"{a}hler threefold with $K_{X} \leq 0$, and let $s$ be a
 $Spin^c$-structure on $X$ with $\deg \xi <0$, where $\xi$ is the
 characteristic line bundle of the $Spin^c$-structure. 
Then $n_{X} (s) = 0$. 
\end{theorem}

The assumption $c_2(X)=0$ is used to obtain the (virtual) dimension of the moduli space to be zero.

For the case where $K_{X} >0$, we get the following.

\begin{theorem}[Theorem \ref{thm:kxp}]
Let $X$ be a compact K\"{a}hler threefold with  $K_{X} >0$.  
Let $s_c$ be the $Spin^c$-structure coming from the complex structure. 
We also assume that $c_2 (X) =0$.  
Then $n_{X}(s_{c}) =1$. 
\end{theorem}

The organisation of this article is as follows. 
In Section \ref{sec:mi}, we briefly describe $Spin^{c}$-structures and the Dirac
operators 
on compact symplectic manifolds,  
and recall the Seiberg--Witten equation on compact symplectic
4-manifolds. 
Then we introduce our equation in six dimensions and describe its linearisation. 
In Section \ref{sec:inv}, we introduce an integer-valued invariant, 
which can be thought of as a 6-dimensional Seiberg--Witten invariant, 
from the moduli space of
solutions to the equation by using Ruan's virtual neighbourhood method. 
We then consider the K\"{a}hler case in Section \ref{sec:kahler}. 
We see that the equations reduce to the vortex type equations in this
case. We then describe the moduli space for the negative degree case. 
We also prove that the moduli spaces are compact when the underlying
manifold is a compact K\"{a}hler threefold,  and
compute the integers in some cases.

\paragraph{Acknowledgements.} 
I would like to thank Dominic Joyce for enlightening discussion on the
subject around the end of 2010, and  for useful comments and 
helpful suggestions on the manuscripts.  
I am grateful to Mikio Furuta for helpful discussion and valuable comments.  
I would like to thank referees for useful comments 
and for pointing out gaps in the earlier versions of the manuscript. 
I am grateful to Oscar  Garc\'{i}a-Prada for sending me some 
useful references. I would also like to thank Hiroshi Iritani and Thomas
Walpuski for useful comments.   
I am also grateful to Seoul National University, NCTS at National Taiwan University, Kyoto University, BICMR at Peking University and Institut des Hautes \'{E}tudes Scientifiques for support and hospitality, where part of this work was done during my visits in 2015--17.

\section{Seiberg--Witten type equations on compact symplectic 6-manifolds}
\label{sec:mi}

\subsection{$\mathbf{Spin^c}$-structure and the Dirac operator on
  compact symplectic manifolds}

A general reference for $Spin^{c}$-structures and its Dirac operator is
Lawson--Michelsohn \cite{LM}.

\paragraph{$\mathbf{Spin^c}$-structure on compact symplectic manifolds.}

Let $X$ be a compact symplectic manifold of dimension $2n$ with symplectic form
$\omega$.  
We fix an almost complex structure $J$ compatible with $\omega$. 
Note we have a metric of the form $g( \cdot, \cdot) := \omega (\cdot , J \cdot )$. 
Any almost complex manifold has the canonical 
$Spin^c$ structure associated with the almost complex structure,   
whose spinor bundle$S = S^+ \oplus S^-$ is given by  $S^{+} = \oplus \Lambda^{0,even} (X)$ and $S^{+} = \oplus \Lambda^{0,odd}
(X)$; and the characteristic line bundle $\xi := \det (S^+)$ is given by $\Lambda^{0,2} = K_{X}^{-1}$.  
Also for any $Spin^c$ structure on X, the spinor bundle $S$ on $X$ can be written as  
$$ S^{+} =  \oplus \Lambda^{0, even}(X) \otimes L, \, S^{-} =
 \oplus \Lambda^{0, odd}(X) \otimes L ,$$
by some line bundle $L$ on $X$.

\paragraph{The Dirac operator on symplectic manifolds.}

The Dirac operator $D_{A'}$ associated to a connection $A'$ on the
characteristic line bundle $\xi$ is
given by the  following composition. 
\begin{equation*}
 \Gamma(S) \xrightarrow{\nabla_{A'}} \Gamma(T^{*} X \otimes S) 
\xrightarrow{g} \Gamma(TX \otimes S) \xrightarrow{\rho} \Gamma(S) , 
\end{equation*}
where $g$ is the metric defined as $g( \cdot ,\cdot ) := \omega (\cdot , J \cdot )$, and $\rho$ is the Clifford multiplication. 
In almost complex case, it is 
written as 
$ D_{A'} = \sqrt{2} ( \bar{\partial}_{A} + \bar{\partial}_{A}^{*})$, 
where $A$ is a connection on $L$.

\subsection{The Seiberg--Witten equations on symplectic 4-manifolds}

We recall the Seiberg--Witten equations in the original form first, 
which was introduced by Witten \cite{W} (see also \cite{M}).    
Let $M$ be a  compact, oriented, smooth 4-manifold. 
We fix a Riemannian metric and a $Spin^c$ structure on $M$.  
We denote by $S^+$ the half spinor bundle over $M$ associated to the
$Spin^c$-structure,  
and by $\xi$ the characteristic line bundle $\det (S^+)$ of the $Spin^c$-structure.

The Seiberg--Witten equations on $M$ are equations seeking for a
connection $A'$ of $\xi$ and a section of $S^{+}$ satisfying the
following. 
\begin{gather*}
D_{A'} \psi  = 0 ,  \quad 
F_{A'}^{+}  = \frac{1}{4} \tau (\psi \otimes \psi^{*} )  , 
\end{gather*}
where $D_{A'}$ is the Dirac operator associated to the connection $A'$, 
and $F_{A'}^{+}$ is the self-dual part of the curvature $F_{A'}$ of the
connection $A'$.   
Also $\tau$ is a  map $\tau : \text{End}\, (S^{+}) \to \Lambda^{+}
\otimes \C$ defined as follows: 
By the Clifford multiplication $\rho: T^*X \to \text{End}(S^+ \oplus S^-)$, 
we define a map $\Lambda^2 \to \text{End} (S^+) $ by 
$\rho ( v \wedge w ) := \frac{1}{2} ( \rho (v) \rho (w) - \rho (w) \rho (v))$. 
Note that  $\rho (v)$ for $v \in T^* X$ is a map from $S^{+}$ to $S^{-}$ or the other way around. 
We then extend this complex linearly to $\Lambda^{+} \otimes \C \to \text{End} (S^+)$. 
The map $\tau$ is defined by the adjoint of this.

We next consider these equations on a compact
symplectic 4-manifold with symplectic structure $\omega$ 
(see \cite{Ta1}, \cite{Ta}, also \cite{BG}, \cite{G2}, \cite{HT} and \cite{K}). 
We fix an almost complex structure compatible with $\omega$.  
Then we have the following decomposition of the self-dual part of the
curvature. 
$F_{A'}^{+} = F_{A'}^{2,0} + F_{A'}^{0}  + F_{A'}^{0,2}$, 
where $F_{A'}^{0}$ is the $\omega$-component of the curvature
$F_{A'}$. 
In addition, we can consider the canonical
$Spin^c$  structure 
whose characteristic line bundle is $K_{M}^{-1} = \Lambda^{0,2} ( T^{*} M \otimes
\C) $. 
Using this canonical $Spin^c$ structure, we can write the half-spinor
bundle $S^+$ for any $Spin^c$ structure as 
$ S^{+} = L \oplus ( L \otimes K_{M}^{-1}) $, 
where $L$ is some complex line bundle on $M$.  
Let $u_0 \equiv 1$ be a constant section of $X \times \C$. 
We then have the canonical $Spin^c$ connection $A_c$ of $K_M^{-1}$ which satisfies $D_{A_c} u_0 = 0$,  
and each  connection $A'$ of the characteristic line bundle is written by 
$A' = A_c + 2 A $, 
where $A$ is a connection of $L$.  
We write a spinor as  $\psi = \varphi_0 u_0 + \varphi_2$, 
where 
$\varphi_0 \in \Gamma (L), \,
\varphi_2 \in \Gamma ( L \otimes K^{-1} )$. 
Then the Seiberg--Witten equations becomes as follows. 
\begin{gather*}
 \dbar_{A} \varphi_0 +\dbar_{A}^{*} \varphi_2 = 0,  \\ 
 F_{A'}^{0,2} =
\frac{\bar{\varphi_0} \varphi_2}{2},  \quad 
 \Lambda F_{A'}^{1,1} = 
- \frac{i}{4} ( |\varphi_2|^2 - |\varphi_0 |^2 ) , 
\end{gather*}
where $\Lambda := (\wedge \omega)^{*}$, also $\bar{\varphi_0} \varphi_2 \in \Gamma (K_{M}^{-1}) = \Omega^{0,2} (X)$.

\subsection{Equations in six dimensions}

Let $X$ be a compact symplectic 6-manifold with symplectic form
$\omega$. 
We fix an almost complex structure $J$ compatible with $\omega$. 
We take a $Spin^{c}$-structure $s$ on $X$, and 
denote by $\xi$ the associated complex line bundle over $X$.

There is a $Spin^c$-structure canonically determined by $J$, which
we denote by $s_c$. 
The corresponding line bundle for $s_c$ is given by the
anti-canonical bundle $K_{X}^{-1}$. 
For a $Spin^c$-structure $s$, 
there is a complex line bundle $L$, and the corresponding line bundle 
for $s$,
which we denote by $\xi$, 
can be written as $\xi = K_{X}^{-1} \otimes L^2$, 
and a connection $A'$ of $\xi$ can be written as 
$A' = A_{c} + 2 A$, 
where $A_{c}$ is a connection of $K_{X}^{-1}$ and $A$ is a connection of $L$.

We put $\mathcal{C} := \mathcal{A} (\xi) \times \Omega^{0,3} (X) \times 
\Omega^{0} (X , L) \times \Omega^{0,2} (X, L)$, where $\mathcal{A}(\xi)$ is the set of all connections on $\xi$.   
We consider the following equations for $(A, u , (\alpha , \beta)) \in \mathcal{C}$. 
\begin{gather}
  \dbar_{A} \alpha + \dbar_{A}^{*} \beta = 0 , \quad 
 \dbar_{A} \beta =  - \frac{1}{2} \alpha u ,  
\label{meq1}\\ 
 F_{A'}^{0,2} + \dbar^{*} u = \frac{1}{4} \bar{\alpha} \beta , \quad 
 \Lambda F_{A'}^{1,1} = - \frac{i}{8} \left( | u |^2 + |\beta|^2 - |\alpha|^2 \right). 
\label{meq2} 
\end{gather}

We call $\mathcal{G} := \Gamma ( X , U(1))$ a  gauge group. This is a
set of  all smooth $U(1)$-valued functions. This group acts on 
solutions to  \eqref{meq1} and \eqref{meq2} by 
\begin{equation*}
A' \mapsto A' - g^{-1} d g , \quad 
u \mapsto u , \quad 
\alpha \mapsto g \alpha , \quad 
\beta \mapsto g \beta , 
\end{equation*}
where $g \in \mathcal{G}$. 
The equations \eqref{meq1} and \eqref{meq2} are equivariant under this
action, 
namely, 
if $(A', u, (\alpha , \beta))$ is a
solution to the equations \eqref{meq1} and  \eqref{meq2}, 
then so is 
$g (A', u, (\alpha , \beta))$ for any $g \in \mathcal{G}$. 
We say two solutions $(A_1' , u_1 , (\alpha_1 , \beta_1)) , \, 
(A_2' , u_2 , (\alpha_2 , \beta_2))$ are {\it gauge equivalent} if there exists a gauge
transformation $g \in \mathcal{G}$ such that 
$(A_1' , u_1 , (\alpha_1 , \beta_1)) = g (A_2' , u_2 , (\alpha_2 ,
\beta_2))$.

As in the Seiberg--Witten case, the stabilizer in $\mathcal{G}$ of 
$(A,u,(\alpha , \beta)) \in \mathcal{C}$ is trivial unless $\alpha
=\beta =0$. 
We then define the following.

\begin{definition}
$(A' , u, (\alpha , \beta ))$ is said to be {\it reducible} if $(\alpha ,
 \beta ) \equiv 0$. It is called {\it irreducible} otherwise.  
\end{definition}

Note that the stabilizer group in the case of reducibles is the group of
constant maps from $X$ to $S^1$, namely, it is $S^1$.

\subsection{Linearisation}
\label{sec:lin}

The differential of the equations at a solution $(A, u , (\alpha, \beta))$ is given by 
\begin{equation*}
\begin{split}
i \Omega^{0} (\R) &\xrightarrow{L_{1}} 
  i \Omega^{1} \oplus \Omega^{0,3} (X) \oplus \left( \Omega^{0,0} (L) \oplus \Omega^{0,2} (L)
 \right)  \\
 & \qquad \qquad \xrightarrow{L_{2}} \left( i \Omega^{2} \cap 
   \left( \Omega^{0,2} \oplus \Omega^{2,0} \right) \oplus i  \Omega^{0}
 \omega \right) \oplus \left( \Omega^{0,1} (L) \oplus \Omega^{0,3} (L)
 \right) ,  \\ 
\end{split}
\end{equation*}
where $L_1$ and $L_{2}$ are 
\begin{equation*}
\begin{split}
L_{1} (ig) &= ( 2 i d g , u g , -i \alpha g , - i \beta) \\ 
L_{2} ( h , \upsilon, (a, b)) 
&= ( P^{+} d (i h) - \frac{1}{8} i \text{Re} ( a \bar{\alpha} ) \omega  
  + \bar{\partial}^{*}  \upsilon + \frac{1}{4} ( \alpha \bar{b} -
 \bar{\alpha} b + \frac{1}{4} ( \beta \bar{a} - \bar{\beta} a ) , \\
 &  \qquad \bar{\partial}_A a + \bar{\partial}_A^{*} b + \pi^{0,1} (i h)
 \alpha / 2 , \bar{\partial}_A b + \pi^{0,3} i h \beta /2 +
 \alpha \upsilon /2  ) .\\  
\end{split}
\end{equation*}
We can discard the zero-th order terms in the above by deforming this by a homotopy. 
Note that deformations which do not change the symbol of the complex maintains the Euler characteristic. 
Such a deformation of the above complex gives us the direct sum of the following two elliptic complexes.  
\begin{equation}
0 \longrightarrow 
  \Omega^{0} (X ,i\R) 
\xrightarrow{d} 
 \Omega^{1} (X , i \R) \oplus  \Omega^{0,3} (X) 
\xrightarrow{P^{+} d + \bar{\partial}^{*}}   
\Omega^{+} (X , i \R) 
\longrightarrow 
0,  
\label{comp:eq1}
\end{equation}
\begin{equation}
0 
\longrightarrow 
\Omega^{0,0} (X, L)
\oplus 
\Omega^{0,2} (X, L) 
\xrightarrow{\bar{\partial}_{A} + \bar{\partial}_{A}^{*}} 
\Omega^{0,1} (X, L)
 \oplus 
\Omega^{0,3} (X, L) 
\longrightarrow 
0. 
\label{AHS2}
\end{equation}
where $\Omega^{+} (X, i\R):= \Omega^{0} (X , i\R ) \omega \oplus 
 \Omega^2 (X , i\R) \cap ( \Omega^{2,0} \oplus \Omega^{0,2})$. 

We then have the following.

\begin{proposition}
The virtual dimension of the moduli space $\mathcal{M}$ is given by 
\begin{equation}
 - \frac{1}{12} c_{1} (X) c_{2} (X) 
 - \frac{1}{24} c_{1} (L) 
\left( 2 c_{1} (X)^2 
 + 2 c_{2} (X) 
  + 6 c_{1} (L) c_{1} (X) 
   + 4 c_{1} (L)^2 \right).
\label{index} 
\end{equation}
\end{proposition}

\begin{proof}
The virtual dimension is the sum of the indices 
of \eqref{comp:eq1} and \eqref{AHS2} with the opposite signs. 
Here, 
the index of \eqref{comp:eq1} can be computed by the following Dolbeault complex. 
\begin{equation}
0 
\longrightarrow 
\Omega^{0,0} (X)
\xrightarrow{\bar{\partial}} 
\Omega^{0,1} (X)
\xrightarrow{\bar{\partial}}
\Omega^{0,2} (X)
\xrightarrow{\bar{\partial}}
\Omega^{0,3} (X)
\longrightarrow 
0 
\label{AHS1}
\end{equation}
by identifying $\Omega^0 (X) \oplus
 \Omega^{0} \omega$ with $\Omega^{0,0}
(X)$; and $\Omega^1 (X) $ with $\Omega^{0,1} (X)$. 
Thus, from the index theorems, we obtain the virtual dimension to be 
$$ - \int_{X} \text{ch} (L) \cdot \text{td} (X) - \int_{X} \text{td} (X)
 . $$ 
As 
$ 
\text{ch} (L) 
= 1 + c_{1} (L) 
 + \frac{1}{2} c_{1} (L)^2 + \frac{1}{6} c_{1} (L)^3 $
and 
$\text{td} (X) 
= 1 + \frac{1}{2} c_{1} (X) + \frac{1}{12} \left( c_{1}(X)^2 + c_{2}
 (X) \right) + \frac{1}{24} c_{1} (X) c_{2} (X)$, we get \eqref{index}.
\end{proof}

\section{Invariant}
\label{sec:inv}

Let $X$ be a compact symplectic 6-manifold with symplectic form
$\omega$. 
We take an almost complex structure compatible with $\omega$, and 
a $Spin^{c}$-structure $c$ on $X$ with the characteristic line bundle $\xi$
being $K_{X}^{-1} \otimes L^2$, where $L$ is a line bundle on $X$.

\paragraph{Moduli space.}
We fix $p >6$ make the following to be smooth. 
Let $\ell$ be a positive integer. 
We consider the following Sobolev completion of the configuration
space. 
$$ \mathcal{C}_{L^{p}_{\ell}} 
:= \mathcal{A}_{L_{\ell}^{p}} (\xi) \times 
L^{p}_{\ell} ( \Lambda^{0,3} ) \times L^{p}_{\ell} ( ( \Lambda^{0,0} \oplus
\Lambda^{0,2}) \otimes L) ,$$
where $\mathcal{A}_{L^{p}_{\ell}} (\xi)$ is the space of $L^{p}_{\ell}$-connections
on $\xi$. 
We also consider $L^{p}_{\ell + 1}$-completion of the space of gauge group
$\mathcal{G} := \text{Map} (X ,U(1))$ 
to get smooth action on the configuration space $\mathcal{C}_{L^{p}_{\ell}}$. 
We then take the quotient 
$$ \mathcal{M}_{L^{p}_{\ell}}  := 
\{ (A, u , (\alpha , \beta)) \in \mathcal{C}_{L^{p}_{\ell}} \, : \, 
(A,u, (\alpha ,\beta)) \text{ satisfies \eqref{meq1} and \eqref{meq2}}
\} / \mathcal{G}_{L^{p}_{\ell+1}}  ,$$
and call it the {\it moduli space} of solutions to the equations
\eqref{meq1} and \eqref{meq2}.

We have the following. 
\begin{proposition}
Fix $p>6$. Let $\ell$ be either a positive integer or $\infty$. 
Let $(A, u, (\alpha , \beta)) \in \mathcal{C}_{L^{p}_{\ell}}$ be a solution to the equations
 \eqref{meq1} and \eqref{meq2}. 
Then there exists a gauge transformation $ g \in
 \mathcal{G}_{L^{p}_{\ell +1}}$
 such that $g (A, u , (\alpha , \beta))$ lies in $L^{p}_{\ell +1}$. 
\label{prop:psm}
\end{proposition}

\begin{proof}
This is a consequence of a local slice theorem \cite[Th.~1.3]{U} (see
 also \cite[Th.~8.1]{WK})
 saying that there exists a smooth connection $A_{0}$ and a gauge
 transformation $g \in \mathcal{G}_{L^{p}_{\ell +1}}$ such that $d_{A}^{*} (g
 (A) - A_{0}) =0$. 
Then the assertion holds since the equations \eqref{meq1} and \eqref{meq2}
 with the above gauge condition from an elliptic system. 
\end{proof}
Hence we abbreviate the Sobolev indices on the moduli space hereafter.

\paragraph{Virtual Neighbourhood. }
We consider the following case. 
\begin{itemize}
\item[(A1)] There are no reducible solutions; and 
\item[(A2)] the moduli space is compact.  
\end{itemize}
Assuming these (A1) and (A2), one can invoke the virtual neighbourhood
method by Ruan \cite{R} to define an integer-valued invariant from
the moduli space. 
As for the condition (A1), we have the following.

\begin{proposition}
Let $X$ be a compact symplectic 6-manifold, and 
let $\xi$ be the characteristic line bundle of a $Spin^{c}$-structure on
 $X$. 
If $\deg \xi < 0$, then there are no reducible solutions to
 the equations. 
\label{prop:red}
\end{proposition}

\begin{proof}
If $\alpha = \beta =0$, then it contradicts $\deg \xi < 0$ as 
$\deg \xi = \frac{1}{16 \pi} \int_{X} 
( |u|^2 + |\beta|^2 - |\alpha|^2  ) vol$. 
Thus, the assertion holds. 
\end{proof}

With regard to the condition (A2), we prove the moduli spaces are compact when the underlying manifold is K\"{a}hler in Section \ref{sec:comp}.

The general setting is as follows. 
Let $\mathcal{B}$ be a smooth Banach manifold, and let $\mathcal{F}$ be a smooth Banach bundle over $\mathcal{B}$. 
Also we consider a nonlinear elliptic differential operator $F: \mathcal{B} \to \mathcal{F}$.

\begin{definition}[\cite{R}]
A triple $(\mathcal{B}, \mathcal{F}, F)$ is said to be a {\it compact smooth triple} if $F^{-1} (0)$ is compact. 
\end{definition}

In our case, we put $\mathcal{B} := \mathcal{C} / \mathcal{G}$, 
$\mathcal{F}:= L^{p}_{1} ( \Lambda^{+} \otimes i \R) \times 
L^{p}_{1} ( (\Lambda^{0,1} \oplus \Lambda^{0,3}) \otimes L )$, 
where $\Lambda^{+} := 
\Lambda^{0} \omega \oplus \left( \Lambda^2 \cap ( \Lambda^{2,0} \oplus
\Lambda^{0,2} ) \right)$ 
and $F : \mathcal{B} \to \mathcal{F}$ defined by the equations 
\eqref{meq1} and \eqref{meq2}. 
Then, if the condition (A1) and (A2) are satisfied, this $( \mathcal{B} , \mathcal{F}, F)$   
forms a compact-smooth triple.

From this compact-smooth triple, one can construct a {\it virtual neighbourhood} $(U, \R^{k} , S)$ of
$\mathcal{M}$  as in \cite{R} with $U$ being a smooth neighbourhood of dimension
$- \text{ind} (L) +k$, where $L$ is the linearised operator of the
equations \eqref{meq1} and \eqref{meq2}, and $k \in \Z_{\ge 0}$ 
such that $\mathcal{M} \times \{ 0 \} \subset U \subset \mathcal{B}
\times \R^{k}$, 
and $S : U \to \R^{k}$ with $S^{-1} (0) =
\mathcal{M}$. 
Here, the orientation of $U$ is given by orienting the determinant line
bundle $\det (L) $ of the linearised operator $L$ from an orientation of 
$H^0 (X, i \R) , H^{1} (X, i\R), H^{0,3} (X)$ and $H^{+} (X, i \R)$. 
One can then define   
a {\it virtual neighbourhood invariant} $\mu_{F}$ in \cite{R} as
follows. 
(i) For $\text{ind}(L)=0$, $\mu_{F}$  is defined to be the algebraic
counting of points in $S^{-1}(y)$ for a regular value $y$; and   
(ii) for $ - \text{ind} (L) >0$, $\mu_{F}$ is defined as 
$\mu_{F} : H^{- \text{ind} (L)} ( \mathcal{B} ,\Z) \to \Z$ 
by $\mu_{F} (\alpha) := \alpha ( [S^{-1} (y)])$ for a regular value $y$, 
where $\alpha \in H^{ - \text{ind} (L)} (\mathcal{B} , \Z)$. 
In \cite[Prop.~2.6]{R}, Ruan proved that the above $\mu_{F}$ is
independent of $y$ and a triple $(U,
\R^k ,S)$.

\paragraph{Invariant.}
We denote by $\mathcal{C}^{*}$ the open subset of $\mathcal{C}$
consisting of irreducible equivalence classes. 
We consider the subgroup $\mathcal{G}_{0}$ of $\mathcal{G}$ consisting
of all gauge transformations which are trivial on the fibre over a fixed
point $x \in X$. 
This is the kernel of the morphism $\mathcal{G} \to S^1$ defined by
evaluating on the fibre over $x$. 
We then consider the quotient $\mathcal{B}^{0} 
:= \mathcal{C}^{*} / \mathcal{G}_{0}$. 
This is the total space of a principal $S^1$-bundle, we denote it by
$\ell$, over $\mathcal{B}^{*}$. 
Then we define an integer $n_{X} (c)$ by 
(i) $\mu_{F}$ if $\text{ind}(L)=0$; (ii) 
$\mu_{F} (c_{1} (\ell)^{-\text{ind}(L)/2})$ 
if $- \text{ind}(L) >0$; and 
(iii) $0$ if $- \text{ind}(L) <0$.

Examples are given in the next section.

\section{Invariants for compact K\"{a}hler threefolds}
\label{sec:kahler}

We describe the equations on compact K\"{a}hler threefolds in Section
\ref{sec:k}. 
In Section \ref{sec:kc}, we describe the moduli spaces for the case $\deg \xi <0$, where 
$\xi$ is the characteristic line bundle for a $Spin^c$-structure on
a compact K\"{a}hler threefold. 
In Section \ref{sec:comp}, we prove that the moduli spaces are compact
if the underlying manifold is a compact K\"{a}hler threefold. 
In Section \ref{sec:kco}, we compute the integers $n_{X} (c)$ defined in 
Section \ref{sec:inv} in some
cases.

\subsection{The equations on compact K\"{a}hler threefolds}
\label{sec:k}

Firstly, we have the following. 
\begin{proposition}
Let $X$ be a compact K\"{a}hler threefold. 
Then the equations \eqref{meq1} and \eqref{meq2} reduce to the
 following. 
\begin{gather}
 \bar{\partial}_{A} \alpha 
  = \bar{\partial}_{A} \beta 
  = \bar{\partial}_{A}^{*} \beta = \alpha u = 0 , 
\label{keq1} \\ 
 F_{A'}^{0,2} = \bar{\partial}^{*} u = \bar{\alpha}\beta = 0  , \quad 
 i \Lambda F_{A'}^{1,1} = \frac{1}{8} \left( |u |^2 + |\beta|^2 - |\alpha|^2 \right)
\label{keq2}
\end{gather}
\label{eqk}
\end{proposition}

\begin{proof}
Using the second equation in \eqref{meq1}, we get 
\begin{equation}
\begin{split}
|| \bar{\partial}_{A} \beta ||^{2}_{L^2} 
 &=  \langle \beta , \bar{\partial}_{A}^{*} \bar{\partial}_{A} \beta \rangle_{L^2} \\
 &=  \frac{1}{2} \langle \beta , - \bar{\partial}_{A}^{*} ( \alpha u) \rangle_{L^2} \\
 &= -  \frac{1}{2} \langle \beta \wedge  \bar{\partial}_{A} \bar{\alpha} , u \rangle_{L^2} 
    -  \frac{1}{2} \langle \beta , (\bar{\partial}_{A}^{*} u) \alpha \rangle_{L^2},  \\
\end{split}
\label{eq:kr1}
\end{equation}
where we used $\bar{\partial}_{A}^{*} = \bar{*} \bar{\partial}_A \bar{*}$ at the last identity. 
The second term in the last line of the above \eqref{eq:kr1} can be
 computed as follows.  
\begin{equation}
\begin{split}
 \langle \beta ,  (\bar{\partial}^{*}  u) \alpha \rangle_{L^2}
&= \langle \beta , -F_{A'}^{0,2} \alpha \rangle_{L^2} 
   + \frac{1}{4} || |\alpha| | \beta | ||^{2}_{L^2} \\ 
&= \langle \beta , - \bar{\partial}_{A} \bar{\partial}_{A} \alpha
 \rangle_{L^2} 
  +  \frac{1}{4}|| |\alpha | | \beta | ||^2_{L^2} \\
&= \langle \beta , \bar{\partial}_{A} \bar{\partial}_{A}^{*} \beta \rangle_{L^2} 
    + \frac{1}{4} || |\alpha| | \beta | ||^{2}_{L^2} \\
&= || \bar{\partial}_{A}^{*} \beta ||^{2}_{L^2} 
+  \frac{1}{4} || |\alpha | | \beta | ||^{2}_{L^2}. \\
 \end{split}
\label{eq:kr2}
\end{equation}
On the other hand, 
from the first equation in \eqref{meq2} and the identity
$\bar{\partial}_{A} F_{A}^{0,2} = 0$ which holds 
for an integrable complex structure, 
we get 
$$\bar{\partial} \bar{\partial}^{*} u = 
\frac{1}{4} ( \bar{\partial}_{A} \bar{\alpha}) \wedge \beta 
+ \frac{1}{4} \bar{\alpha} \bar{\partial}_{A} \beta . 
$$  
From this we obtain 
\begin{equation}
\begin{split}
|| \bar{\partial}^{*} u ||^{2}_{L^2} 
 &= \langle u ,  \bar{\partial} \bar{\partial}^{*} u \rangle_{L^2} \\
 &= \frac{1}{4} \langle u , (\bar{\partial}_{A} \bar{\alpha}) \wedge \beta \rangle_{L^2}  
     + \frac{1}{4} \langle u , \bar{\alpha} \bar{\partial}_{A} \beta \rangle_{L^2} \\
 &=  \frac{1}{4} \langle u , (\bar{\partial}_{A} \bar{\alpha}) \wedge \beta \rangle_{L^2} 
     - \frac{1}{8} || |\alpha| |u| ||^{2}_{L^2}.  
\end{split}
\label{eq:kr3}
\end{equation}
Hence, from \eqref{eq:kr1}, \eqref{eq:kr2} and \eqref{eq:kr3}, 
we get  
\begin{equation*}
 || \bar{\partial}_{A} \beta ||^{2}_{L^2}  
  + 2 || \bar{\partial}^{*} u ||^{2}_{L^2} 
  + \frac{1}{4} || | \alpha | | u | ||^{2}_{L^2} 
  + \frac{1}{2} || \bar{\partial}_{A}^{*} \beta ||^{2}_{L^2}
  + \frac{1}{8} || | \alpha | | \beta | ||^{2}_{L^2}  
  =0 . 
\end{equation*}
Thus, the assertion holds. 
\end{proof}

We define the degree of $\xi$ by 
$$\deg \xi := c_{1} (\xi) \cdot [\omega^2] 
 = \frac{i}{2 \pi} \int_{X} F_{A'} \wedge \omega^2 . 
$$

\begin{proposition}
Let $X$ be a compact K\"{a}hler threefold, and let $\xi$ be the
 characteristic line bundle of a $Spin^c$-structure on $X$. 
Let $(A,u, (\alpha ,\beta))$ be a solution to the equations \eqref{meq1}
 and \eqref{meq2}. Then the following holds. 
\begin{enumerate}
\item[$\rm (i)$]
If $\deg \xi < 0$, then $\beta \equiv 0$ and $u \equiv 0$. 
\item[$\rm (ii)$] 
If $\deg \xi > 0$, then $\alpha \equiv 0$. 
\item[$\rm (iii)$] 
If $\deg \xi =0$, then $\alpha \equiv 0$, $\beta \equiv 0, u \equiv 0$. 
\end{enumerate}
\label{prop:dc}
\end{proposition}

\begin{proof}
From $\bar{\alpha} \beta = 0$, either $\alpha$ or $\beta$ is zero on
 some open subset of $X$, thus 
either $\alpha$ or $\beta$ is zero on the whole of $X$ by unique
 continuation as 
$\bar{\partial}_{A} \alpha =0$ and 
$\bar{\partial}_{A} \beta = \bar{\partial}_{A}^{*} \beta =0$. 
Similarly, from $\alpha u =0$, either $\alpha$ or $u$ is zero on an
 open set in $X$, so either $\alpha$ or $u$ is zero on $X$ again by
 unique continuation as $\bar{\partial}_{A} \alpha =0$  and 
 $\bar{\partial}^{*} u =0$. 
On the other hand, from the second equation in \eqref{keq2}, we have 
\begin{equation}
 \deg \xi = \frac{i}{2 \pi} \int_{X} F_{A'} \wedge \omega^2 
 = \frac{1}{16 \pi} \int_{X} \left( |u|^2 + | \beta |^2 - | \alpha |^2 \right)
 vol. 
\label{eq:degeq}
\end{equation}

Firstly, we consider the case $\deg \xi < 0$. 
In this case, because of \eqref{eq:degeq}, $\alpha \equiv 0$ contradicts $\deg \xi <0$, 
thus we have $\alpha \not\equiv 0$. 
Then, from the above reasoning in the top of this proof, we get $\beta \equiv 0,  u \equiv 0$.

Secondly, we consider the case $\deg \xi > 0$. 
In this case, if $\beta \equiv 0$ and $u \equiv 0$, we get a
 contradiction again from \eqref{eq:degeq}. 
Thus, $\beta \not\equiv 0$ or $u \not\equiv 0$, and therefore $\alpha \equiv 0$.

Finally, we consider the case $\deg \xi =0$. 
If $\alpha \not\equiv 0$, then we get $\beta \equiv 0 , u \equiv 0$; 
and this results in  $\deg \xi < 0$. 
Hence $\alpha \equiv 0$. 
Then, as $\alpha \equiv 0$,  again from \eqref{eq:degeq}, we get $\beta \equiv 0$ and $u \equiv
 0$. 
\end{proof}

From Proposition \ref{prop:dc} (iii) above, 
we immediately get the following.

\begin{proposition}
Let $X$ be a compact K\"{a}hler threefold, 
and let $\xi$ be the
 characteristic line bundle of a $Spin^c$-structure on $X$. 
Let $(A,u, (\alpha ,\beta))$ be a solution to the equations \eqref{meq1}
 and \eqref{meq2}. Then, if $\deg \xi =0$, the moduli space is isomorphic to $H^1 (X , \R) / H^{1} (X,
 \Z)$. 
\end{proposition}

\begin{proof}
From Proposition \ref{prop:dc} (iii), we get $\alpha \equiv \beta \equiv
 u \equiv 0$, thus the equations reduce to the Hermitian--Einstein equations. 
Then, from the Hitchin--Kobayashi correspondence of the Hermitian--Einstein
 connection for line bundles, 
the moduli space $\mathcal{M}$ is identified with the moduli space of holomorphic
 structures on $\xi$. 
Since $\xi$ is topologically trivial, it is then isomorphic to $H^1 (X , \R) /
 H^1 (X , \Z)$.  
\end{proof}

\subsection{The moduli space for the negative degree case} 
\label{sec:kc}

Let $X$ be a compact K\"{a}hler threefold. 
In this subsection, we describe the moduli space $\mathcal{M}$ of
solutions to the equations \eqref{keq1} and \eqref{keq2} for the case $\deg \xi < 0$, where 
$\xi$ is the characteristic line bundle of a $Spin^c$-structure on
$X$. 
In this case, from Proposition \ref{prop:dc}, 
we have $\beta \equiv 0$ and $u \equiv 0$. 
Thus the equations reduce to the vortex equations (see \cite{B} and
\cite{G1}). 
We  have the following description of the moduli space, 
which was originally obtained by Thomas \cite{T} in the context of the
Seiberg--Witten equations on compact K\"{a}hler threefolds.

\begin{proposition}[\cite{T}, Th.~2.6]
Let $X$ be a compact K\"{a}hler threefold, and let $c$ be a
 $Spin^c$-structure on $X$ with $\deg \xi <0$, 
where $\xi$ is the characteristic line bundle of the
 $Spin^c$-structure. 
Then the moduli space of solutions to the equations
 \eqref{keq1} and \eqref{keq2} can be 
 identified with $\bigcup_{\xi} \mathbb{P} H^{0} 
\left( X, (K_{X} \otimes \xi )^{1/2} \right)
$, where the union is taken through all holomorphic structures on
 $\xi$. 
\label{prop:mneg}
\end{proposition}

\begin{proof}
This is derived from a result by Bradlow \cite[Th.~4.3]{B}. 
For the original Seiberg--Witten case, the corresponding result was
 described by Witten \cite{W} (see also \cite[\S 2]{FM}). 

Firstly, we fix a hermitian metric $k$ on $\xi$ and a section $\alpha \in
 \Gamma (X, L)$. 
We then vary the hermitian metric by $e^{u}$, where $u$ is
 a real-valued function. 
The induced unitary connection on $\xi$ can be written as $A_{k} + 2
 \partial u$, where $A_{k}$ is a unitary connection on $\xi$ induced from the metric $k$, 
and the second equation in \eqref{keq2} becomes  
\begin{equation}
\Delta u + \frac{1}{8} | \alpha |_{k}^2 e^{2u} = 
- i \Lambda F_{A_k}, 
\label{eq:bv}
\end{equation}
where $\Delta$ is the negative definite Laplacian on functions. 
From the assumption that $\deg \xi < 0$, 
we have 
$\int_{X} - i \Lambda F_{A_k} d V > 0$ 
and also $\frac{1}{8}|\alpha|_{k}$ is
 strictly positive somewhere. 
Thus, as in \cite{B}, 
we can invoke results by Kazdan--Warner \cite{KW} to deduce that there exists a unique
 solution $u$ to the equation \eqref{eq:bv}. 
Hence, for each fixed non-trivial section of $\left( K_{X} \otimes \xi
 \right)^{1/2}$, we have a solution to the equations \eqref{keq1} and
 \eqref{keq2}.

On the other hand, two sections differed by a non-zero constant are
 gauge equivalent solutions; and 
two gauge equivalent solutions $\alpha$ and $\alpha'$ represent the same
 point in $\mathbb{P} H^{0} 
\left( X, (K_{X} \otimes \xi )^{1/2} \right)$ 
since holomorphic automorphisms of a holomorphic line bundle consist of
 only non-zero constant functions. 
Hence the assertion holds. 
\end{proof}

\paragraph{Poincar\'{e} invariants.}
There is an algebraic counterpart of the Seiberg--Witten invariants,
 called {\it Poincar\'{e} invariants}, by D\"{u}rr, Kabanov and Okonek
 \cite{DKO}. 
These are defined through the Hilbert schemes of divisors on a smooth
 projective variety. 
They constructed a perfect obstruction theory in the sense of Behrend
 and Fantechi \cite{BF} on the Hilbert schemes under some assumptions,   
and defined the Poincar\'{e} invariants as virtual intersection
 numbers on the Hilbert schemes, which are consequently deformation
 invariants. 
The assumptions are satisfied for smooth projective surfaces, and the invariants were
 proved to be equivalent to the Seiberg--Witten invariants (see
 \cite{DKO}, 
 \cite{CK} for more detail).

The invariants can be defined other than the surfaces. 
In fact, D\"{u}rr, Kabanov and Okonek constructed the perfect obstruction theory under the
 assumption  that $H^{i} (\mathcal{O}_{D} (D)) =0$ for all $i > 1$ and
 for 
 any point $D$ in the Hilbert scheme of divisors on a smooth
 projective variety with a fixed topological type \cite[Th.~1.12]{DKO}. 
(Note that the obstruction space is $H^{1} (\mathcal{O}_{D} (D))$.) 
The above assumption is satisfied in the case we are describing in this
 subsection (i.e. the case for $\deg \xi < 0$) 
for example with $X$ being a compact Calabi--Yau or Fano threefold as below.  
(This was also mentioned in \cite[Remark~1.13]{DKO}.)

We  consider the following short exact sequence. 
$$ 
0  \longrightarrow \mathcal{O}_{X} \longrightarrow 
\mathcal{O}_{X} (D)  \longrightarrow  \mathcal{O}_{D} (D)
 \longrightarrow  0 . 
$$
From this, we get 
\begin{equation}
\cdots \longrightarrow  
H^{1} (\mathcal{O}_{X} (D)) \longrightarrow  
H^{2} (\mathcal{O}_{D} (D)) \longrightarrow  
H^{2} (\mathcal{O}_{X}) \longrightarrow 
H^{2} (\mathcal{O}_{X} (D)) \longrightarrow  \cdots . 
\label{les}
\end{equation}
Here, we think of $\mathcal{O}_{X} (D)$ as $L$. If $K_{X} \cong \mathcal{O}$ or $K_{X} <
 0$ and $\deg
 \xi < 0$, then $\deg L < 0$ as $L^2 = K_{X} \otimes \xi$. 
Thus from the Kodaira vanishing theorem, we obtain 
$H^{i} (\mathcal{O}_{X} (D)) =0$ for $i < 3$. 
Hence, from this and \eqref{les}, we get 
$$ H^{2} (\mathcal{O}_{D} (D)) \cong H^{2} (\mathcal{O}_{X})  . $$
From the Serre duality, we have $H^2 (\mathcal{O}_{X}) \cong H^1 (
 K_{X})^{\vee}$. 
Hence, if $X$ is a compact Calabi--Yau or Fano threefold, $H^{2}
 (\mathcal{O}_{D} (D)) =0$. 
(This is deduced from $H^{1} (K_{X}) \cong  H^{1} (\mathcal{O}_{X}) =0$ for the Calabi--Yau
 case, and from the Kodaira vanishing theorem for the Fano case.) 
Therefore, one can construct the Poincar\'{e} invariants of 
D\"{u}rr, Kabanov and Okonek in these cases. 
\qed

\subsection{Compactness of the moduli space}
\label{sec:comp}

In this subsection, we prove that the moduli spaces of solutions to the
equations  \eqref{meq1} and  \eqref{meq2} are compact when the underlying manifold is a
compact K\"{a}hler threefold.  
In order to do that, we first introduce the following rescaled
equations.

\paragraph{Rescaled equations.}
Let $X$ be a compact K\"{a}hler threefold. 
We introduce a scaling parameter $s>0$ to the equation \eqref{keq1} and
\eqref{keq2} as follows. 
\begin{gather}
 \bar{\partial}_{A} \alpha 
  = \bar{\partial}_{A} \beta 
  = \bar{\partial}_{A}^{*} \beta = \alpha u = 0 , 
\label{rkeq1} \\ 
 F_{A'}^{0,2} = \bar{\partial}^{*} u = \bar{\alpha}\beta = 0  , \quad 
 i \Lambda F_{A'}^{1,1} = \frac{s}{8} \left( |u |^2 + |\beta|^2 -
 |\alpha|^2 \right) .
\label{rkeq2}
\end{gather}

Similar rescaled equations for the Seiberg--Witten equations on compact
symplectic four-manifolds were
introduced by Taubes in \cite{Ta1} (see also \cite{Ta}), 
and the ones for the vortex equations and its asymptotic of
solutions as $s \to \infty$ were studied by Baptista \cite{Ba1}, \cite{Ba2} 
and Liu \cite{L}.

We define the moduli space of solutions to the equations
\eqref{rkeq1} and \eqref{rkeq2}, which we denote by $\mathcal{M}(s)$, in the same way as we do that for the
equations \eqref{meq1} and \eqref{meq2}. 
Note that the presence of the parameter $s$ does not affect the
statements thus far. 
We prove the following below.

\begin{proposition}
Let $X$ be a compact K\"{a}hler threefold. 
Then the moduli space $\mathcal{M} (s)$ is compact if $s$ is sufficiently
 large.   
\label{prop:comp}
\end{proposition}

Since $(A , u, (\alpha , \beta)) \mapsto 
(A , \sqrt{s} u, \sqrt{s} (\alpha , \beta))$ 
gives a bijection between $\mathcal{M} (s)$ and $\mathcal{M} = \mathcal{M}(1)$, 
we get the following.

\begin{Corollary}
Let $X$ be a compact K\"{a}hler threefold. 
Then the moduli space $\mathcal{M}$ of solution to the equations \eqref{keq1}
 and \eqref{keq2} is compact.  
\end{Corollary}

We prove Proposition \ref{prop:comp} in three steps (I), (II) and (III) below. 
In the rest of this subsection, constants denoted by $C$ or $D$ do not
depend on the parameter $s$.

\paragraph{(I) Bounds on $|\alpha|$, $|\beta|$ and $|u|$.}
We prove $L^{\infty}$-bounds on $\alpha , \beta$ and $u$.

Firstly, we consider the case $\deg \xi < 0$ case. 
In this case, from Proposition \ref{prop:dc}, 
we have $\beta \equiv 0$ and $u \equiv 0$.  
We also have the
following point-wise estimate on $\alpha$.

\begin{lemma}
Let $X$ be a compact K\"{a}hler threefold, 
and let $\xi$ be the
 characteristic line bundle of a $Spin^c$-structure on $X$. 
Let $(A,u, (\alpha ,\beta))$ be a solution to the equations \eqref{rkeq1}
 and \eqref{rkeq2}. Then, 
if $\deg \xi <0$ then $ | \alpha | \leq C / \sqrt{s}$, where $C$ is a positive
 constant which depends upon the curvature of the canonical connection
 on $K_{X}^{-1}$. 
\label{lem:esta}
\end{lemma}

\begin{proof}
For $\alpha \in \Omega^{0,0} (X, L)$, we have the following from the
 K\"{a}hler identity (see e.g. \cite[\S~2.2]{K} or \cite[\S~6.1.3]{DK}). 
\begin{equation}
 \bar{\partial}_{A}^{*} \bar{\partial}_{A} \alpha 
 = \frac{1}{2} \nabla_{A}^{*} \nabla_{A} \alpha 
 - \frac{1}{2} i ( \Lambda F_{A}) \alpha . 
\label{eq:weita}
\end{equation}
Then, using the equations \eqref{keq1} and \eqref{keq2}, we obtain  
$$ 
0 = \frac{1}{2} \nabla_{A}^{*} \nabla_{A} \alpha 
 - \frac{1}{2} \left( - \frac{1}{2} F_{K_{X}^{-1}} - \frac{s}{8} | \alpha
 |^2 \right) \alpha .$$

On the other hand, we have $\frac{1}{2} \Delta | \alpha |^2 = 
\langle \nabla_{A}^{*} \nabla_{A} \alpha , \alpha \rangle - | \nabla_{A}
 \alpha |^2$, where $\Delta$ is the Laplacian on functions. 
Thus, we get 
$$ \frac{1}{2} 
 \Delta | \alpha |^2 + | \nabla_{A} \alpha |^2 
= - \frac{1}{2} F_{K_{X}^{-1}} |\alpha|^2 
 - \frac{s}{8} |\alpha|^4 .$$
From the maximum principle, if $x_0$ is a local maximum of $| \alpha |^2
 (x)$, as the Laplacian is non-negative at a maximum, we obtain 
$$ - \frac{1}{2} F_{K_{X}^{-1}} (x_0) |\alpha|^2 (x_0)  
 - \frac{s}{8} |\alpha|^4 (x_0) \geq 0 . $$
Thus, we get either $|\alpha |^2 (x_0) = 0$ or $|\alpha|^2 (x_0) \leq -
 4 F_{K_{X}^{-1}} (x_0) /s $. 
Hence the assertion holds. 
\end{proof}

If $\deg \xi >0$, from Proposition \ref{prop:dc}, 
we have $\alpha \equiv 0$. 
In this case, 
similar to Lemma \ref{lem:esta}, we have the following point-wise
estimate on $\beta$.

\begin{lemma}
Let $X$ be a compact K\"{a}hler threefold, 
and let $\xi$ be the
 characteristic line bundle of a $Spin^c$-structure on $X$. 
Let $(A,u, (\alpha ,\beta))$ be a solution to the equations \eqref{rkeq1}
 and \eqref{rkeq2}. Then, 
if $\deg \xi > 0$ then $ | \beta | \leq C / \sqrt{s} $, 
where $C$ is a positive constant which depends upon the curvature of 
the canonical connection on $K_{X}^{-1}$ and of a 
 connection on $\Lambda^{0,2}$. 
In addition, there exists a constant $D >0$, which depends on the metric
 of the underlying manifold $X$ and $\deg \xi$ such that 
$|u | \leq D /s$. 
\label{lem:estbu}
\end{lemma}

\begin{proof}
For $\beta \in \Omega^{0,2} (X, L)$, 
we have the following again from the
 K\"{a}hler identity. 
\begin{equation}
\bar{\partial}_{A}^{*} \bar{\partial}_{A} \beta 
 +   \bar{\partial}_{A} \bar{\partial}_{A}^{*} \beta  
 = \frac{1}{2} \tilde{\nabla}_{A}^{*} \tilde{\nabla}_{A} \beta  
 +  \frac{1}{2} i ( \Lambda (  F' + F_{A} ) ) \beta , 
\label{eq:weitb}
\end{equation}
where $\tilde{\nabla}_{A}$ is the unitary connection of $\Lambda^{0,2}
 \otimes L$, and $F' + F_{A} $ is the curvature of $\tilde{\nabla}_{A}$. 
Then, using the equations \eqref{rkeq1} and \eqref{rkeq2}, we obtain 
$$  
 0 = \frac{1}{2} \tilde{\nabla}_{A}^{*} \tilde{\nabla}_{A} \beta  
 +  \frac{1}{2} \left( i(\Lambda F') - \frac{1}{2} F_{K_{X}^{-1}} + \frac{s}{8}
 ( |u|^2 + |\beta|^2 ) \right) \beta ,$$
Thus, again using the maximum principle, if $x_1$ is a local maximum of
 $| \beta |^2 (x)$, then we get 
$$  -  \frac{1}{2} \left( i(\Lambda F') - \frac{1}{2} F_{K_{X}^{-1}}
 \right) | \beta |^2 (x_1) - \frac{s}{16} \left( 
  |u|^2 (x_1) + |\beta|^2 (x_1)  \right) | \beta |^2 (x_1) \geq 0 .  $$
Thus, either $| \beta |^2 (x_1) = 0$ or 
$| \beta |^2 (x_1) + | u |^2 (x_1) \leq  \left(  - 4 i \Lambda F' 
+ 4 F_{K_{X}^{-1}} \right) / s $. 
Hence we get $|\beta| \leq C / \sqrt{s} $ for a constant $C >0$.

We next prove the bound on $u$. 
From the Weitzenb\"{o}ck formula, $u$ satisfies 
$ \langle \nabla \nabla^{*} u , u \rangle \leq  C |u|^2 $, where 
$C >0$ is a constant which depends on the metric of the underlying
manifold $X$. 
Hence, from a maximum principle, we get 
$ | u | \leq C \int_{X} |u|^2 $. 
On the other hand, from \eqref{eq:degeq}, $s \int_{X} |u|^2$ is bounded
 by $\deg \xi$ in this case. 
Thus the assertion holds. 
\end{proof}

\paragraph{(II) Estimate on the curvature.}

We next prove the following estimate on the curvature.

\begin{proposition}
Let $(A, u, (\alpha , \beta))$ a solution to the equations \eqref{rkeq1}
 and \eqref{rkeq2}. Then 
there exists a constant $C > 0$ such that 
$|F_{A}| \leq C$ if $s$ is large enough. 
\label{lem:bfa}
\end{proposition}

This type of estimate was obtained by Taubes \cite[Prop.~2.4]{Ta}
in the case of the Seiberg--Witten equations on compact symplectic
four-manifolds 
(see also Kanda \cite[Prop.~6.1]{Ka}). 
Our proof goes in a similar way to those in \cite{Ta} and \cite{Ka}
except that we have the extra terms coming from $u$ in the equations.

From Proposition \ref{prop:dc}, either $\alpha$ vanishes 
or $\beta$ and $u$ do if $\deg \xi \neq 0$. 
But, in the following, 
in order to avoid the proof becoming longer, we deal with the two cases
simultaneously, namely, 
$\alpha, \beta$ and $u$ appear at the same time.

\vspace{0.2cm}

\hspace{-0.75cm}
{\it Proof of Proposition \ref{lem:bfa}.} 
We write $F_{A'}^{+} := F_{A'}^{2,0} + F_{A'}^{0} + F_{A'}^{0,2}$, where
$F_{A'}^{0}$ is the $\omega$-component of the curvature, 
and denote by $F_{A'}^{\perp}$ the orthogonal component of $F_{A'}$ to
$F_{A'}^{+}$ so that $F_{A'} = F_{A'}^{+} + F_{A'}^{\perp}$. 
An $L^{\infty}$-bound for $F_{A'}^{+}$ comes from the equation \eqref{rkeq2}
and the $L^{\infty}$-bounds on $\alpha , \beta$ and $u$ in Lemmas
\ref{lem:esta} and \ref{lem:estbu}. 
We thus prove an $L^{\infty}$-bound for $F_{A'}^{\perp}$ below.

From the Bianchi identity, we have 
$$ d F_{A'}^{+} + d F_{A'}^{\perp} = 0 .$$
Thus, by the Weitzenb\"{o}ck formula, we obtain 
\begin{equation}
\frac{1}{2} \nabla^{*} \nabla F_{A'}^{\perp} 
 + R F_{A'}^{\perp} 
= - P^{\perp} d^* d F_{A'}^{+}, 
\label{wb}
\end{equation}
where $P^{\perp}$ is the orthogonal projection to the compliment of the self-dual part, 
and $R \in \text{Hom} (\Lambda^{\perp} , \Lambda^{\perp})$ depends on
 the metric of the underlying manifold. 
From \eqref{rkeq1} and \eqref{rkeq2}, the right hand side of \eqref{wb}
 becomes the following. 
\begin{equation*}
\begin{split}
P^{\perp} 
d^{*} d \left\{ - \frac{i s}{8} 
\left( |\alpha|^2 - |\beta|^2 \right) \right\} 
&= - \frac{is}{8} P^{\perp} \left( \partial^{*} \partial + 
\bar{\partial}^{*} \bar{\partial} \right) 
\left( | \alpha |^2 - | \beta |^2  - |u|^2 \right) \omega \\ 
&= - \frac{s}{4} P^{\perp} \left( 
 \bar{\partial} \partial | \alpha |^2 + \partial \bar{\partial} |\beta|^2 
+ \partial \bar{\partial} |u|^2 |  \right). 
\end{split}
\end{equation*}
Hence \eqref{wb} becomes 
\begin{equation*}
\begin{split}
\frac{1}{2} \nabla^{*} &\nabla F_{A'}^{\perp} 
 + R F_{A'}^{\perp}   \\
&= - \frac{s}{4} \langle F_{A}^{\perp} \alpha , \alpha \rangle_{L} 
 - \frac{s}{4} \langle \left( F_{A}^{\perp} + F_{\Lambda^{0,2}} \right)
 \beta 
, \beta \rangle_{\Lambda^{0,2} \otimes L}  
 - \frac{s}{4} \langle 
F_{K_{X}^{-1}} u , u \rangle_{K_{X}^{-1}} \\
& \qquad  + \frac{s}{4} 
P^{\perp} 
\left\{ \langle d_{A} \alpha , d_{A} \alpha \rangle_{L} 
 + \langle d_{\tilde{\nabla}_{A}} \beta , d_{\tilde{\nabla}_{A}} \beta
 \rangle_{\Lambda^{0,2} \otimes L} 
 +  \langle d_{\nabla} u , d_{\nabla} u 
 \rangle_{K_{X}^{-1}} \right\}. 
\end{split}
\end{equation*}
Taking the inner product of the above with $F_{A}^{\perp}$, and using 
$( \Delta |F| ) |F| \leq \langle \nabla^{*} \nabla F , F \rangle$, 
we obtain the following. 
\begin{equation*}
\begin{split}
&\left( 
\frac{1}{2} \Delta 
+ \frac{s}{4} ( | \alpha |^2 + |\beta|^2 + 1)  
\right) |F_{A'}^{\perp}|  \\
&  \qquad  \leq \left( |R| + \frac{s}{4} \right) |F_{A'}^{\perp}| + \frac{s}{4} |\beta|^2 |F_{\Lambda^{0,2}}
 | + \frac{s}{4} |u|^2 |F_{K_{X}^{-1}}|\\
& \qquad \qquad + \frac{s}{4} 
\left| P^{\perp} 
\left\{ \langle d_{A} \alpha , d_{A} \alpha \rangle_{L} 
 + \langle d_{\tilde{\nabla}_{A}} \beta , d_{\tilde{\nabla}_{A}} \beta
 \rangle_{\Lambda^{0,2}} 
+  \langle d_{\nabla} u , d_{\nabla} u \rangle 
 \rangle_{K_{X}^{-1}}
\right\}\right| \\ 
&  \qquad \leq \left( |R| + \frac{s}{4} \right) |F_{A}^{\perp}| + C s |\beta|^2 + C s |u|^2 
 + \frac{s}{4 \sqrt{2}} 
\left( | \nabla_{A} \alpha |^2 + | \tilde{\nabla}_{A} \beta |^2  + |
 \nabla u|^2 \right) . \\ 
\end{split}
\end{equation*}
Putting $\phi_{0} := \frac{s}{4 \sqrt{2}} \left( |\alpha|^2 + |\beta|^2
 + |u|^2 + 1 \right)$ and $R' := |R| + \frac{s}{4}$, we get 
\begin{equation} 
\left( \frac{1}{2} \Delta 
+ \frac{s}{4} (  | \alpha |^2 +  |\beta|^2 +1 ) \right) 
\left( |F_{A'}^{\perp} | - \phi_{0} \right) 
\leq R' |F_{A'}^{\perp} | .
\label{eq:phi0}
\end{equation}

We next consider a solution $\phi$ to the following equation. 
\begin{equation}
\left( \frac{1}{2} 
\Delta + \frac{s}{4} 
\left( |\alpha|^2 + |\beta|^2 + 1 \right) \right) \phi 
= R' |F_{A'}^{\perp} |. 
\label{eq:phi}
\end{equation} 
From a maximal principle with \eqref{eq:phi0} and \eqref{eq:phi}, we
 get 
\begin{equation}
|F_{A'}^{\perp}| \leq \phi_{0} + \phi . 
\label{eq:phiphi0}
\end{equation}

We next bound $|| \phi ||_{L^2 (X)}$ by $\sup_{X} \phi$. 
Multiplying $\phi$ to \eqref{eq:phi}, and adding $- \frac{s}{4} 
\left( |\alpha|^2 + |\beta|^2 \right) \phi^2$, we get 
$$  
\frac{1}{4} \Delta  \phi^2 
+ \frac{1}{2} |\nabla \phi |^2 
+ \frac{s}{4} | \phi |^2 
= R' |F_{A'}^{\perp} | \phi 
 - \frac{s}{4} \left( |\alpha|^2 + |\beta|^2  \right) \phi^2 .$$
Then, using H\"{o}lder inequality to each term of the right hand side of
 the above equation, we get 
$$ \frac{1}{4} 
\Delta \phi^2 + \frac{s}{12} 
|\phi|^2 \leq \frac{6}{s} R'|F_{A'}^{\perp} |^2 
+ \frac{3}{16} s \left( |\alpha|^2 + |\beta|^2 \right)^2  \, \left(
 \sup_{X} \phi \right)^2 . 
$$
Hence we obtain  
$$
\int_{X} | \phi |^2 \leq C \left\{ \frac{1}{s^2} 
\int_{X} |F_{A'}^{\perp} |^2 
 + \frac{1}{s} \left( \sup_{X} \phi \right)^2 \int_{X} s \left(
 |\alpha|^2 + |\beta|^2 \right)^2 \right\}. $$ 
Here, we have 
$$ \int_{X} \left\{ |F_{A'}^{+}|^2 - |F_{A'}^{\perp}|^2 \right\} 
= \langle c_{1}^{3} (\xi) , [X] \rangle . $$
Therefore we get 
$$ || \phi ||_{L^2 (X)} 
\leq \frac{C}{s} + \frac{C}{s} \sup_{X} \phi . $$

We next introduce a function $\psi \in L^{2}_{2}$ which satisfies 
\begin{equation}
\left( \frac{1}{2} 
\Delta + \frac{s}{4} \left( |\alpha|^2 + |\beta|^2 +1 \right) \right) \psi 
 = R' |F_{A'}^{\perp} | 
\left( |\alpha|^2 + |\beta|^2 + 1 \right) . 
\label{eq:psi}
\end{equation}
By using a maximum principle for $\psi - \frac{4}{s} R' \sup_X |F_A^{\perp}|$, where $R_{0} := \sup_{X} R'$, 
we get 
\begin{equation} \sup_{X} \psi 
\leq \frac{4}{s} R_0 \sup_{X} |F_{A'}^{\perp} | . 
\label{eq:psib}
\end{equation}
From \eqref{eq:phi} and \eqref{eq:psi}, we have 
$$ 
\left( \frac{1}{2} \Delta 
 + \frac{s}{4} \left( |\alpha|^2 + |\beta|^2 + 1 \right) \right) 
(\phi - \psi ) 
= R' |F_{A'}^{\perp} | \left( |\alpha|^2 + |\beta|^2 \right) . $$

As in \cite{Ka}, we use the following form of the maximum principle by
Gilberg--Trudinger 
(see \cite[Th.~6.6]{Ka}).

\begin{theorem}[\cite{GT}]
Let $D:= - a^{ij} \frac{\partial}{\partial x_{i}}
 \frac{\partial}{\partial x_{j}}
+ b^{j}  \frac{\partial}{\partial x_{j}} +c$ be an elliptic operator
 defined on the unit ball in $\R^{6}$ and satisfy the following
 condition. 
\begin{enumerate}
\item[$(1)$] Let $A$ be a symmetric matrix $(a^{ij})_{ij}$. There exist 
	     positive constants $\lambda_1 \geq \lambda_2$ such that 
$\lambda_1 |\xi|^2 \geq \xi^{t} A \xi \geq \lambda_2 |\xi|^2$ for all
	     $\xi \in \R^6$. In other word, $A$ is uniformly positive
	     definite. 
\item[$(2)$] There exists $\lambda_3 \geq 0$ such that $|b| \leq
	     \lambda_3$ and $c \geq - \lambda_3$.
\end{enumerate}
Then there exists a positive constant $C'$ which depends only on
 $\lambda_1, \lambda_2$ and $\lambda_3$ with the following
 significance. 
If $f \in C^{2} (\overline{B}_{1})$ obeys the differential inequality
 $Df \leq g$ on $B_{1}^{+} \subset B_1$, then it follows that 
$$\sup_{B_{1/2}} f \leq C' \left( || f ||_{L^{2} (B^{+}_{1})} 
 + || g ||_{L^6 (B_{1}^{+})} \right), $$
where $B_{1}^{+}$ denotes the subset $\{ x \in B_1 : f(x) \geq 0\}$. 
\label{th:GT}
\end{theorem}

Let $x \in X$ be the point which attains the maximum of $\phi -\psi$. 
Then from Theorem \ref{th:GT}, we get 
\begin{equation*}
\begin{split} 
\sup_{X} (\phi - \psi) 
&= \sup_{B_{1/2} (x)} (\phi - \psi ) \\
& \leq C \left\{ || \phi ||_{L^2 (X)} 
  + || \psi ||_{L^2 (X)}  \right. \\
& \qquad  + \sup_{X} |F_{A'}^{\perp} | \cdot 
 || \, |\alpha|^2 + |\beta|^2 ||_{L^6 (X)} \} \\ 
& \leq C 
\left\{ \frac{C}{s} + \frac{C}{s} \left( \sup_{X} \phi \right) 
 + \frac{C}{s} \sup_{X} |F_{A'}^{\perp} | \right\} . 
\end{split}
\end{equation*}
On the other hand, from \eqref{eq:psib}, we have 
$$ \sup_{X} \phi - \frac{C}{s} \sup_{X} |F_{A'}^{\perp} | 
 \leq \sup_{X} \left( \phi - \psi \right) . $$
Thus, 
$$ \sup_{X} \phi - \frac{C}{s} \sup_{X} |F_{A'}^{\perp} | 
 \leq
 C 
\left\{ \frac{C}{s} + \frac{C}{s} \left( \sup_{X} \phi \right) 
 + \frac{C}{s} \sup_{X} |F_{A'}^{\perp} | \right\} . $$
So, 
$$ \left( 1 - \frac{C}{s} \right) \sup_{X} \phi
\leq \frac{C}{s} + \frac{C}{s} \sup_{X} |F_{A'}^{\perp}
$$
Hence, if $s$ is large enough, we get 
$$ \sup_{X} \phi \leq C \left( 
\frac{1}{s} + \frac{1}{s} \sup_{X} |F_{A'}^{\perp} | \right) .$$
Therefore, from \eqref{eq:phiphi0}, we get 
$$ \sup_{X} |F_{A'}^{\perp}| \leq C ,$$
if $s$ is large enough. 
Thus the assertion holds. 
\qed

\paragraph{(III) Proof of Proposition \ref{prop:comp}.}

Once the $L^{\infty}$-bound on the curvatures is obtained, the rest of
the argument will be rather  standard. 
The following argument is modeled on ones by \cite[Th.~3.2]{CGMS} and \cite[pp.~61--67]{WJ}.

From Proposition \ref{lem:bfa}, the curvatures $F_{A_{i}}$ are uniformly 
bounded in $L^{p}$ for $p>1$. 
Thus, from \cite[Th.~1.5]{U} (see also \cite[Th.~A]{WK}), there exist gauge transformations $g_{i}
\in \mathcal{G}_{L^{p}_{2}}$ and a constant $C >0$ such that 
$$ ||g_{i} (A_{i} - A_{0}) ||_{L^{p}_{1}} \leq C . $$ 
We also have an $L^{p}_{2}$-bound on $g_{i} (\alpha_{i} , \beta_{i})$ as 
$(\alpha_{i} , \beta_{i})$ has $L^{\infty}$-bound. 
Hence we get a weak limit $(A, u, (\alpha , \beta ))$ in $L^{p}_{1}$ 
after taking a subsequence if necessary. 
From Proposition \ref{prop:psm},  there exists a gauge transformation $g \in
\mathcal{G}_{L^{p}_{2}}$ such that $g ( A, u, (\alpha ,\beta))$ is
smooth. 
Hence we suppose that the limit $(A,u , (\alpha , \beta))$ is smooth.

We then invoke the local slice theorem \cite[Th.~1.3]{U} (see also \cite[Th.~8.1]{WK}), namely, we
find $h_{i} \in \mathcal{G}_{L^{p}_{2}}$ with 
$d_{A}^{*} (h_{i} ( A_{i}) - A ) =0 $ for each $i$ and a constant $C > 0$
such that the following hold for each $i$. 
\begin{gather*}
|| h_{i} (A_{i}) - A ||_{L^{\infty}} \leq 
C || A_{i} - A ||_{L^{\infty}} , \quad 
|| h_{i} (A_{i}) - A ||_{L^{p}_{1}} \leq 
C || A_{i} - A ||_{L^{p}_{1}} . 
\end{gather*}
In addition, $\{ h_{i} \}$ is uniformly bounded in $L^{p}_{2}$. 
Then we replace $(A_{i} , u_{i} , (\alpha_{i}, \beta_{i}))$ by 
$h_{i} (A_{i} , u_{i} , (\alpha_{i}, \beta_{i}))$. 
For this sequence, we have the gauge condition $d_{A}^{*} (A_{i} - A)
=0$ with a uniform $L^{p}_{1}$-bound with $L^{p}_{2}$-gauges. 
Thus, the elliptic estimate yields the following. 
\begin{gather*}
||A_{i} - A ||_{L^{p}_{\ell}} \leq C_{\ell} , \quad
|| (u_{i} , (\alpha_{i} , \beta_{i} ) ) ||_{L^{p}_{\ell}} 
\leq C_{\ell}
\end{gather*}
for each $\ell \geq 1$ and all $i$, 
where $C_{\ell}$'s are positive constants. 
Hence the assertion holds. 
\qed

\subsection{Some computations}
\label{sec:kco}

In the previous subsection, we proved that the moduli spaces are compact
in the K\"{a}hler case. 
Hence the assumption (A2) in Section \ref{sec:inv} is satisfied in this case.   
Regarding the assumption (A1) in Section \ref{sec:inv} on reducibles, 
we have Proposition \ref{prop:red} for compact symplectic 6-manifolds.  
In addition, 
the following holds if the underlying manifolds are K\"{a}hler
threefolds.

\begin{proposition}
Let $X$ be a compact K\"{a}hler threefold, and 
 threefold, and 
let $\xi$ be the characteristic line bundle of a $Spin^{c}$-structure on
 $X$. 
If $\deg \xi > 0$ and $K_{X} <0$, then there are no reducible
 solutions to the equations. 
\end{proposition}

\begin{proof}
If $\deg \xi >0$, then $\alpha \equiv 0$ from Proposition
\ref{prop:dc}. 
Since we assume that $K_{X} <0$, we have $u \equiv 0$. 
Suppose now for a contradiction that $\beta \equiv 0$. 
Then we get $\deg \xi =0$ again from \eqref{eq:degeq}. 
This is a contradiction. Hence $\beta \not\equiv 0$. 
\end{proof}

Therefore, in these cases, one can define the integers $n_{X} (c)$ in
Section \ref{sec:inv} from
the moduli space $\mathcal{M}$. 
We compute some of them below. 
These are analogies of those for the Seiberg--Witten invariants,
presented as Proposition 7.3.1 in \cite{M}.

Firstly, we have the following.

\begin{proposition}
Let $X$ be a compact K\"{a}hler threefold with $K_{X} \leq 0$, and let $c$ be a
 $Spin^c$-structure on $X$ with $\deg \xi <0$, where $\xi$ is the
 characteristic line bundle of the $Spin^c$-structure.  
Then there are no solutions to the equation \eqref{keq1} and
 \eqref{keq2}. Namely, the moduli space is empty in this case. 
\label{prop:emp}
\end{proposition}

\begin{proof}
Suppose for a contradiction that there is a solution 
$(A, u, (\alpha , \beta))$ to the equation \eqref{keq1} and
 \eqref{keq2}. 
Since $\deg \xi < 0$, we get $\beta \equiv 0$ and $u \equiv 0$ from 
Proposition \ref{prop:dc}. 
As $\alpha$ is a holomorphic section of $L$, we get $\deg L > 0$. 
However, this contradicts the fact that 
$L^2 = K_{X} \otimes \xi$ has negative degree. 
Thus, the assertion holds. 
\end{proof}

Hence, we get the following.

\begin{Corollary}
Let $X$ be a compact K\"{a}hler threefold with $K_{X} \leq 0$, and let $c$ be a
 $Spin^c$-structure on $X$ with $\deg \xi <0$, where $\xi$ is the
 characteristic line bundle of the $Spin^c$-structure. 
Then $n_{X} (c) = 0$.  
\label{cor:kxn}
\end{Corollary}

For the case $K_{X}>0$, we have the following.

\begin{theorem}
Let $X$ be a compact K\"{a}hler threefold with $c_2 (X) =0$. 
Let $s_c$ be the $Spin^c$-structure coming from the complex structure. 
Assume that $K_{X} >0$. 
Then $n_{X}(s_{c}) =1$. 
\label{thm:kxp} 
\end{theorem}

\begin{proof}

Since $s_c$ is the $Spin^c$-structure coming from the complex structure, 
the corresponding line bundle $\xi$ is $K_{X}^{-1}$. 
As we assume that $K_{X} >0$, thus $\deg \xi <0$; and $\beta \equiv 0$ and $u
 \equiv 0$. 
Because $L$ is trivial in this case, we only have a solution $(A_{0},
 \alpha_{0})$, 
where $A_{0}$ is the canonical connection of  $\xi$, 
and $\alpha_0$ is a non-zero, constant section of $L$. 
Thus the moduli space $\mathcal{M}$ contains only a single point.

We then prove that the moduli space $\mathcal{M}$ is smooth, and its
 dimension is zero.  
Firstly, the index \eqref{index} vanishes, since 
we assume that $c_2 (X) =0$, and $c_1 (L) = 0$ as $L$ is trivial. 
Thus, the dimension of the moduli space is zero, if it is smooth.

We next prove that the moduli space is actually smooth. 
The proof goes in a similar way of that presented in
 \cite[pp.~119--122]{M} for the Seiberg--Witten case except that we have
 the extra terms coming from $u$ in the equations.

Firstly, recall that the following elliptic complex of the
 Atiyah--Hitchin--Singer type in Section \ref{sec:lin}. 
\begin{equation*}
\begin{split}
i \Omega^{0} (\R) &\xrightarrow{L_{1}} 
  i \Omega^{1} \oplus \Omega^{0,3} (X) \oplus \left( \Omega^{0,0} (L) \oplus \Omega^{0,2} (L)
 \right)  \\
 & \qquad \qquad \xrightarrow{L_{2}} \left( i \Omega^{2} \cap 
   \left( \Omega^{0,2} \oplus \Omega^{2,0} \right) \oplus i  \Omega^{0}
 \omega \right) \oplus \left( \Omega^{0,1} (L) \oplus \Omega^{0,3} (L)
 \right) . \\ 
\end{split}
\end{equation*}
We denote by $H^0, H^1 , H^2$ the cohomology of the above complex. 
We now consider the above $L_1$ and $L_{2}$ at $(A_0, 0 , (\alpha_0 , 0))$. 
Then they become as
 follows. 
\begin{equation*}
\begin{split}
L_{1} (ig) &= ( 2 i d g , 0, -i \alpha_0 g , 0) \\ 
L_{2} ( h , \upsilon, (a, b)) 
&= ( P^{+} d (i h) - \frac{1}{8} i \text{Re} ( a \bar{\alpha_0} ) \omega  
  + \bar{\partial}^{*}  \upsilon + \frac{1}{4} ( \alpha_0 \bar{b} -
 \bar{\alpha_0} b) , \\
 &  \qquad \bar{\partial} a + \bar{\partial}^{*} b + \pi^{0,1} (i h)
 \alpha_0 / 2 , \bar{\partial} b  +
 \alpha_0 \upsilon /2  ) .\\  
\end{split}
\end{equation*}
Firstly, since $\alpha_0$ is a non-zero constant section, the kernel of $L_{1}$ is
 trivial. 
Thus, $H^0 =0$.

We next assume that $L_{2} ( h , \upsilon , (a,b)) =0$. 
Then, from the second component of $\bar{\partial} L_{2} ( h , \upsilon , (a,b)) =0 
$, we get 
$ \bar{\partial} \bar{\partial}^{*} b + 
 \frac{1}{2} \bar{\partial} ( \pi^{0,1} ( i h ) \alpha_0 ) = 0 $, 
where we used $\bar{\partial} \bar{\partial} = 0$ and $\bar{\partial}
 \alpha_0 =0$.

On the other hand, we have 
$ \bar{\partial} ( \pi^{0,1} ( i h)) = ( d (i h ))^{0,2} = -
 \bar{\partial}^{*} \upsilon + \frac{\bar{\alpha_0}b}{2}$. 
In addition, from the third component of $\bar{\partial}^{*} 
L_{2} (h , \upsilon , (a,b))$, we obtain 
$\bar{\partial}^{*} \bar{\partial} b + \frac{1}{2}
 \alpha_0 
 \bar{\partial}^{*} \upsilon =0$. 
Thus, we get 
$$ \bar{\partial} \bar{\partial}^{*} b + \bar{\partial}^{*}
 \bar{\partial} b + \frac{1}{4} | \alpha_0 |^2 b = 0 . $$  
By taking $L^2$-inner product of this with $b$, 
we get $|| \bar{\partial}^{*} b ||^{2}_{L^2} + || \bar{\partial} b
 ||^{2}_{L^2} + \frac{1}{4} ||
 \bar{\alpha_0} b ||^{2}_{L^2} = 0$. 
Thus, $\bar{\partial}^{*} b = \bar{\partial} b = 0, \bar{\alpha_0} b = 0$. 
As $\alpha_0 \neq 0$, we get $b
 \equiv 0$ by the unique continuation.

We next write $ih = \bar{\chi} - \chi$ by some $\bar{\chi} \in
 \Omega^{0,1} 
(X , \C)$. 
Then $L_{2} ( i h , \upsilon , (a,b)) = 0$ becomes 
\begin{gather}
P^{+} ( \partial \bar{\chi} - \bar{\partial} \chi ) - \frac{i}{2} 
 \text{Re} ( a \bar{\alpha_0}) \omega + \bar{\partial}^{*} \upsilon = 0, 
\label{eq:l22} \\
\bar{\partial} a + \frac{1}{2} \bar{\chi} \alpha_0 = 0, \quad \alpha_0
 \upsilon =0. 
\label{eq:l23} 
\end{gather}
We now write $a = ( p + iq) \alpha_0$, where $p$ and $q$ are real-valued
 functions. 
Then by adding $L_a (i q )$ to $(i h , \upsilon , (a, b))$, 
we can make $a = p \alpha_0$. 
By this, the first equation of \eqref{eq:l22} is now 
$$ \bar{\partial} ( p \alpha_0) + \frac{1}{2} 
 \bar{\chi} \alpha_0 = 0 .$$
From this, $\bar{\chi} = -2 \bar{\partial} ( p \alpha_0)$. 
Putting this into \eqref{eq:l23}, we get 
$4 \Delta p + p =0$. 
Thus, we get $p = 0$. 
Hence $i h = 0$ and $\upsilon =0$. 
Thus, the kernel of $L_{2}$ is in the image of $L_1$. 
Therefore $H^1 =0$.

As we mentioned in the second paragraph of this proof, the index of the elliptic complex is zero; and since
 $H^0 = H^1 =0$ as computed above, we conclude that $H^2 =0$. 
Thus, $n_{X} (s_c) = \pm 1$.

We then determine the sign of the number $n_{X} (s_c)$. 
As in \cite[p.121]{M}, we decompose the operator $L_1$ and $L_{2}$ as 
$L_{1} = \overline{L}_{1} +M_{1}$,  $L_{2} = \overline{L}_{2} +M_{2}$,
 where $\overline{L}_{1} = (2d ,0)$, $ \overline{L}_{2} = (P^{+} d + \bar{\partial}^{*} , 
 \bar{\partial} + \bar{\partial}^{*})$, 
and $M_1$ and $M_2$ are operators of order zero. 
We denote by $\mathcal{H}^{i}$ the space of harmonic forms defined by
 $\overline{L}_{i}$. 
Then the orientation of the determinant bundle of the elliptic complex
 is given by that of $\mathcal{H}^{i}$'s, 
which are induced by the given orientations on $H^{i} (X, i \R)$ and
 $H^{0,i} (X)$. 
Hence the sign of the determinant bundle at the solution is that of the
 following complex. 
\begin{equation*}
\begin{split}
0  \longrightarrow H^{0} (X , i \R ) 
   \longrightarrow &
\begin{matrix}
H^{1} (X , i \R ) \oplus H^{0,3} (X , \C) \\
 \oplus \\
H^{0,0} (X, \C) \oplus H^{0,2} (X , \C) \\
\end{matrix}  \\
& \qquad 
\qquad  \qquad  \longrightarrow 
\begin{matrix}
H^{0,2} (X , \C ) \oplus H^{0} (X , i \R) \omega  \\
 \oplus \\
H^{0,1} (X, \C) \oplus H^{0,3} (X , \C) \\
\end{matrix} 
\longrightarrow 0 . \\
\end{split}
\end{equation*}
In the above, the maps $H^{1} (X , i\R) \to H^{0,1} (X , \C)$, 
$H^{0,3} (X) \to H^{0,3} (X, \C)$ and $H^{0,2} (X , \C) \to H^{0,2} (X ,
 \C)$ are orientation preserving isomorphisms induced by $M_2$. 
Thus, by factoring through these isomorphisms, we get 
$$ 0 \longrightarrow i \R \longrightarrow \C \longrightarrow i \R  
\longrightarrow 0 ,$$
where the first map is minus the inclusion and the second map is $i /8$
 times the real part. 
Hence, with the given orientation, the sign of the determinant of this
 complex is $+ 1$. 
\end{proof}


\begin{flushleft}
Department of Mathematics, Osaka University  \\
1-1 Machikaneyama-cho, Toyonaka, Osaka, 560-0043, Japan \\
yu2tanaka@gmail.com
\end{flushleft}



\addcontentsline{toc}{chapter}{Bibliography}
\begin{thebibliography}{CDFN83}


\bibitem[Ba1]{Ba1}J. M. Baptista, 
{\it On the $L^2$ metrics of vortex moduli spaces}, 
Nucl. Phys. B {\bf 844}, (2010), 308--333.  


\bibitem[Ba2]{Ba2} 
J. M. Baptista, 
{\it Moduli spaces of Abelian vortices on K\"{a}hler manifolds}, arXiv:1211.0012 


\bibitem[BF]{BF}
K. Behrend and B. Fantechi, 
{\it The intrinsic normal cone}, 
Invent. math. {\bf 128} (1997), 45--88. 

\bibitem[Br]{B}
S. B. Bradlow, 
{\it Vortices in holomorphic line bundles over closed K\"{a}hler manifolds},
	Comm. Math. Phys.  {\bf 135}  (1990), 1--17. 


\bibitem[BG]{BG}
S. B. Bradlow and O. Garc\'{i}a-Prada, 
{\it Non-abelian monopoles and vortices}, 
  Geometry and physics (Aarhus, 1995),  567--589, 
Lecture Notes in Pure and Appl. Math. {\bf 184} Dekker, New York, 1997. 


\bibitem[CK]{CK}
H. Chang, Y. Kiem, 
{\it Poincar\'{e} invariants are Seiberg--Witten invariants}, 
Geom. Topol. {\bf 17} (2013), 1149--1163. 


\bibitem[CGMS]{CGMS}K. Cieliebak, A. R. Gaio, I. Mundet i Riera 
and D. A. Salamon, 
{\it The symplectic vortex equations and invariants of Hamiltonian group
	      actions},  
J. Symplectic Geom.  {\bf 1}  (2002), 543--645. 


\bibitem[DK]{DK}S. K. Donaldson and  P. B. Kronheimer, 
\textit{The Geometry 
	     of Four-Manifolds}, Oxford University Press, 1990.  


\bibitem[DKO]{DKO}
M. D\"{u}rr, A. Kabanov and C. Okonek, 
{\it Poincar\'{e} invariants}, Topology {\bf 46} (2007), 225--294. 


\bibitem[FM]{FM}
R. Friedman and  J. W. Morgan, 
{\it Algebraic surfaces and Seiberg-Witten
	    invariants},  
J. Algebraic Geom. {\bf 6} (1997), 445--479.


\bibitem[G1]{G1}
O. Garc\'{i}a-Prada, 
{\it Invariant connections and vortices}, Comm. Math. Phys. {\bf 156}
	    (1993),  527--546.


\bibitem[G2]{G2}
O. Garc\'{i}a-Prada, 
{\it Seiberg-Witten invariants and vortex equations}, 
Sym\'{e}tries quantiques (Les Houches, 1995),  885--934, North-Holland, Amsterdam, 1998. 


\bibitem[GT]{GT}D. Gilbarg and N. S. Trudinger, \textit{Elliptic
	       partial 
	    differential equations of second order}, 2nd edition, 
	    Springer-Verlag, 1983. 
 

\bibitem[HT]{HT}
M. Hutchings and C. H. Taubes, 
\textit{An introduction to the Seiberg-Witten equations on symplectic
	    four-manifolds}, 
Symplectic geometry and topology (Park City, UT, 1997), 
103--142, Amer. Math. Soc., 1999.


\bibitem[Ka]{Ka}Y. Kanda, \textit{The monopole equations and
	$J$-holomorphic curves on weakly convex almost K\"{a}hler
	4-manifolds}, Trans. Amer. Math. Soc. \textbf{353} (2001), 
	2215--2243.   

\bibitem[Ko]{K}D. Kotschick, 
{\it The Seiberg--Witten invariants of symplectic four-manifolds (after
	   C. H. Taubes)}, 
S\'{e}minaire Bourbaki, Vol 1995/96.   
Ast\'{e}risque  {\bf 241} (1997), 195--220. 


\bibitem[KW]{KW}
J. L. Kazdan and F. W.  Warner, 
{\it Curvature functions for compact 2-manifolds}, 
Ann. of Math.  {\bf 99}  (1974), 14--47. 


\bibitem[L]{L}C. Liu, 
{\it Dynamics of Abelian vortices without common zeros in the adiabatic
	   limit}, 
Comm. Math. Phys.  {\bf 329}  (2014), 169--206.


\bibitem[LM]{LM}H. B. Lawson, Jr. and M-L.  Michelsohn, 
         {\it Spin
         geometry}, Princeton University Press, Princeton NJ, 1989. 


\bibitem[M]{M}J. W. Morgan, {\it The Seiberg--Witten equations and
	applications to the topology of smooth of smooth four-manifolds}, 
    Mathematical Notes, {\bf 44} Princeton University Press NJ, 1996. 


\bibitem[R]{R}Y. Ruan, 
{\it Virtual neighborhoods and monopole equations} in ``Topics in
	   Symplectic 4-manifolds'', Int. Press, Cambridge, MA, 1998. 


\bibitem[Ta1]{Ta1}C. H. Taubes, 
{\it The Seiberg-Witten invariants and symplectic forms}, 
Math. Res. Lett. {\bf 1}  (1994),  809--822.


\bibitem[Ta2]{Ta}C. H. Taubes, $SW \Rightarrow Gr$: \textit{From the
        Seiberg-Witten equations to pseudo-holomorphic curves},
        J. Amer. Math. Soc. \textbf{9} (1996), 845--918.


\bibitem[Th]{T}R. P. W. Thomas, {\it The Seiberg--Witten equations on
	   complex 3-folds}, First Year dissertation,
	   University of Oxford,
	   1995.  


\bibitem[U]{U}
K. K. Uhlenbeck, {\it 
Connections with $L^{p}$ bounds on curvature}, 
Comm. Math. Phys.  {\bf 83} (1982), 31--42.  


\bibitem[WJ]{WJ}
J. Wehrheim, 
{\it Vortex invariants and toric manifolds}, Dissertation, 
Ludwig-Maximilians-Universit\"{a}t M\"{u}nchen, 2008, 
arXiv: 08120299. 


\bibitem[WK]{WK}
K. Wehrheim, 
{\it Uhlenbeck compactness}, 
EMS Series of Lectures in Mathematics. European Mathematical Society,
	    Zurich, 2004. 


\bibitem[W]{W}E. Witten, {\it Monopoles and 4-manifolds},
	Math. Res. Letters {\bf 1} (1994), 764--796. 
\end{thebibliography}
\end{document}